\documentclass[12pt]{amsart}
\usepackage{amssymb,amsmath,amsthm}
\usepackage{a4}
\usepackage{enumitem}
\usepackage{fullpage}
\usepackage{mathrsfs}
\usepackage{color}
\usepackage[colorlinks=true,
linkcolor=webgreen,
filecolor=webgreen,
citecolor=webgreen]{hyperref}
\definecolor{webgreen}{rgb}{0,0,1}
\definecolor{recrown}{rgb}{1,.2,.6}

\usepackage{fullpage}
\usepackage{orcidlink}
\begin{document}
\newtheorem{theorem}{Theorem}
\newtheorem{corollary}[theorem]{Corollary}
\newtheorem{lemma}[theorem]{Lemma}
\newtheorem{Proposition}[theorem]{Proposition}
\theoremstyle{definition}
\newtheorem*{examples}{Examples}
\newtheorem{example}{Example}
\newtheorem{conjecture}[theorem]{Conjecture}
\newtheorem{thmx}{\bf Theorem}
\renewcommand{\thethmx}{\text{\Alph{thmx}}}
\newtheorem{lemmax}{Lemma}
\renewcommand{\thelemmax}{\text{\Alph{lemmax}}}
\theoremstyle{definition}
\newtheorem*{definition}{Definition}
\theoremstyle{remark}
\newtheorem*{remark}{\bf Remark}
\newcommand{\re}{\mathrm{Re}}
\title{\bf A study of some recent irreducibility criteria for polynomials having integer coefficients}
\markright{}
\author{Sanjeev Kumar$^{\dagger}$ {\large \orcidlink{0000-0001-6882-4733}}}
\address{$~^\dagger$ Department of Mathematics, SGGS College, Sector-26, Chandigarh-160019, India}
\email{sanjeev\_kumar\_19@yahoo.co.in}
\author{Jitender Singh$^{\ddagger,*}$ {\large \orcidlink{0000-0003-3706-8239}}} 
\address{$~^\ddagger$ Department of Mathematics, Guru Nanak Dev University, Amritsar-143005, India}
\email{sonumaths@gmail.com}
\date{}
\maketitle
\parindent=0cm
\footnotetext[3]{$^{,*}$Corresponding author: jitender.math@gndu.ac.in}
\footnotetext[2]{sanjeev\_kumar\_19@yahoo.co.in}
\begin{abstract}
In this article, we give an account of some recent irreducibility testing criteria for polynomials having integer coefficients over the field of rational numbers.
\end{abstract}
\section{Introduction}
The problem of testing irreducibility of polynomials over a designated field appears to be deceptively simple. The problem is in fact baroque and classical. There exist several interesting and elegant irreducibility testing criteria, an account of some of these can be found in \cite{D,T}. In the present article, we shall concentrate primarily on the irreducibility of polynomials having integer coefficients, which are based on elementary divisibility properties of integers and study of location of zeros of respective polynomials in the complex plane. Comprising of seven sections, this exposition congregates traditional as well as recent Eisenstein-Sch\"onemann-Dumas type irreducibility criteria, canvases Schur-type, P\'olya-type irreducible polynomials, and presents factorization results based on fundamental notions such as Newton polygons, root location, truncation of binomials, and polynomial shifting.

Let us quickly recall that a polynomial, $f=a_0+a_1x+\cdots+a_nx^n$ having integer coefficients is primitive if  its content is unity, that is,  $\gcd(a_0, a_1, \ldots, a_n)=1$. We begin our study by quoting one of the most prolific irreducibility criterion due to Eisenstein which is stated as follows:
\begin{theorem}[Eisenstein \cite{E}]\label{1}
Let $f=a_0+a_1x+\cdots+a_nx^n\in \mathbb Z[x]$. Let $p$ be a prime number such that
\begin{enumerate}[label=$(\roman*)$]
\item $p\nmid a_n$,
\item $p\mid a_i$ for each $i=0, 1, \ldots, n-1$,
\item $p^2\nmid a_0$.
\end{enumerate}
Then the polynomial $f$ is irreducible over $\mathbb Q$.
\end{theorem}
An exciting application of Theorem \ref{1} is the deduction of the irreducibility of the cyclotomic polynomial
 \begin{eqnarray*}
\varphi_p(x)=1+x+x^2+\cdots+x^{p-1},
 \end{eqnarray*}
 wherein $p$ is a prime number, the first simple proof of which was given by Gauss \cite[Article 341]{G}. Contrary to one's expectation, Theorem \ref{1} is not directly applicable to $\varphi_p(x)$ but to its translate $\varphi_p(x+1)$. It must be mentioned here that the irreducibility of  $\varphi_p(x)$ follows from an equivalent version of Eisenstein irreducibility criterion, inspired from Gauss's \emph{Disquisitiones Arithmeticae}, which is credited to Sch\"{o}nemann,  was overlooked by Eisenstein himself. The method involved is the reduction of coefficients of $f$ modulo $p$ and the computation takes place over the finite field $\mathbb F_p$. More precisely, we have the following result whose proof develops an acquaintance with the calculations involved in the modulo $p$ reduction process.
\begin{theorem}[Sch\"{o}nemann \cite{S}]\label{2}
Let $f\in \mathbb Z[x]$ be a monic polynomial with $f(0)\neq 0$. Let $p$ be a prime and $n$ be a positive integer such that for $g, h\in \mathbb Z[x]$, we have
\begin{enumerate}[label=$(\roman*)$]
\item $f(x)=g(x)^n+p h(x)$,
\item $\overline{g}=g\mod p$ is irreducible over $\mathbb{F}_p$,
\item $\overline{g}\nmid \overline{h}$ in $\mathbb{F}_p[x]$.
\end{enumerate}
Then $f$ is irreducible over $\mathbb{Q}$.
\end{theorem}
\begin{proof}
In view of the fact that $f$ is monic, it is clear that $\deg f=m=n\deg g$. We may assume without loss of generality that $g$ is monic. On the contrary assume that $f(x)=f_1(x)f_2(x)$, where $f_1, f_2\in \mathbb{Z}[x]$ such that  $\max\{\deg f_1, \deg f_2\}<\deg f$. Consequently, we have
\begin{eqnarray*}
\overline{f}(x)=f(x~\mbox{mod}){p}=g(x)^n~\mbox{mod}{p}=\overline{f_1}(x)\cdot \overline{f_2}(x).
\end{eqnarray*}
So,  $f_1(x)=g(x)^r+p g_1(x), f_2(x)=g(x)^{n-r}+pg_2(x)$ for some polynomials $g_1, g_2\in \mathbb Z[x]$. We then have
\begin{eqnarray*}
g(x)^n+ph(x)=f(x)=f_1(x)f_2(x)=g(x)^n+p(g_1(x)g(x)^{n-r}+g_2(x)g(x)^r+pg_1(x)g_2(x)),
\end{eqnarray*}
which yields that $h(x)=g_1(x)g(x)^{n-r}+g_2(x)g(x)^r+pg_1(x)g_2(x)$. This shows that $\bar{h}(x)=\overline{g_1}(x)\overline{g}(x)^{n-r}+\overline{g_2}(x)\overline{g}(x)^{r}$, and so, $\overline{g} \mid \overline{h}$, which contradicts the condition $(iii)$ of the hypothesis.
\end{proof}
Theorem \ref{2} yields Eisenstein irreducibility criterion for $g(x)=x$, $a_n=1$, and $ h(x)=a_0+a_1x+\cdots+a_{n-1}x^{n-1}=f(x)-x^n$. For a remarkable historical account of ``why Eisenstein proved the Eisenstein criterion and why Sch\"{o}nemann discovered it first$?$" the reader is referred to the exciting exposition \cite{C}.
\section{Generalizations of Eisenstein and Sch\"{o}nemann criteria}
An appealing generalization of Eisenstein and Sch\"{o}nemann irreducibility criteria follows via valuation theoretic notions in assembly with the concept of Newton polygons.  So, let us swiftly recollect the associated notions. Let $K$ be a field and $\Gamma$, a totally ordered abelian group (additively written). A valuation on $K$ is a map $v:K\rightarrow \Gamma\cup \{\infty\}$ such that $v(0)=\infty$, and
\begin{eqnarray*}
v(xy) = v(x) + v(y); \min\{(v(x), v(y)\}\leq v(x + y)~\text{for all}~x,y\in K,
\end{eqnarray*}
where $\leq$ denotes the total order of $\Gamma$. The pair $(K,v)$ is called a valued field with the
value group $v(K^{\times})$  of $v$ and the valuation ring $R=\{x\in K~|~v(x)\geq 0\}$ of $v$. The valuation ring $R$  is called discrete valuation ring if the value group $v(K^{\times})$ is isomorphic to $\mathbb{Z}$.

For example, for a prime  $p$, the map $v_p:\mathbb{Q}\rightarrow \mathbb{R}\cup\{\infty\}$  defined by  $v_p(a/b)=\alpha$ for all $a/b\in\mathbb{Q}$, where $a/b=p^\alpha c/d$ wherein $p\nmid cd$ is a valuation on the field of all rational numbers $\mathbb{Q}$ called the $p$-adic valuation of $\mathbb{Q}$ with the value group $\mathbb{Z}$ and the discrete valuation ring $\mathbb{Z}$. 

In the context, one of the earliest known criterion is due to Dumas \cite{Du} which is stated as follows.
\begin{theorem}[Dumas \cite{Du}]\label{Du}
Let $(K, v)$ be a valued field with discrete valuation ring $R$. Let $f=a_0+a_1x+\cdots+a_nx^n\in R[x]$ be  such that
\begin{enumerate}[label=$(\roman*)$]
\item $v(a_n)=0$,
\item $\dfrac{v(a_i)}{n-i}>\dfrac{v(a_0)}{n}$ for  $0\leq i\leq n-1$,
\item $\gcd(v(a_0), n)=1$.
\end{enumerate}
Then the polynomial $f$ is irreducible over $K$.
\end{theorem}
It must be pointed out in the league of such generalizations that recently, Weintraub \cite{W} gave a generalization of  Eisenstein's criterion and also furnished a correction to the incorrect claim made by Eisenstein himself \cite{E}. 
More specifically, the following exquisite criterion was discovered.
\begin{theorem}[Weintraub \cite{W}]\label{4}
Let $f(x)=a_0+a_1x+\cdots+a_nx^n$ be a polynomial with integer coefficients and $p$ be a prime satisfying the following conditions:
\begin{enumerate}[label=$(\roman*)$]
\item $p|a_i$ for each $i=0,1,\ldots,n-1$, $p\nmid a_n$,
\item there exists an index $k\in \{0, 1, \ldots, n-1\}$ such that $p^2\nmid a_k$,
\item $k_0=\min\{k~|~p^2\nmid a_{k}\}$.
\end{enumerate}
If $f(x)=g(x)h(x)$ for $g,h\in \mathbb Z[x]$, then $\min\{\deg g, \deg h\}\leq  k_0$. In particular, if $f$ is primitive and $k_0 = 0$, then $f$ is irreducible. If $k_0 = 1$ and $f$ has no zero in $\mathbb{Q}$, then $f$ is irreducible.
\end{theorem}
Theorem \ref{4} has the following application. In view of the fact that an irreducible polynomial equation of prime degree is solvable by radicals if and only if each of its roots can be expressed as a rational function of any two of them \cite[Proposition VIII]{Ga}, it appears vivid that if $p\geq 5$ is a prime number and $g= x^p -p^px + p$ and $h = x^p - p2^px + p^2$, then neither $g$ nor $h$ is solvable by radicals (See Weintraub \cite{W}).

Yet another powerful tool for establishing irreducibility is via Newton polygons. We recall briskly that
a Newton polygon of a polynomial $f=a_0+a_1x+\cdots+a_nx^n\in\mathbb{Z}[x]$ with respect to the $p$-adic valuation $v_p$ of $\mathbb{Q}$ is defined as the lower convex-hull of the points $(0,v_p(a_0)),\ldots, (n,v_p(a_n))$ in the Euclidean plane. For example, the Newton polygon of $f$ satisfying the hypothesis of Theorem \ref{1} consists of single straight line segment joining the points $(0,1)$ and $(n,0)$, which by Theorem \ref{Du} ensures the irreducibility of $f$.

Using the method of Newton polygons,  Bonciocat \cite{Bonciocat2015} obtained several alluring Sch\"onemann-Eisenstein-Dumas type irreducibility criteria for polynomials having integer coefficients by imposing some divisibility conditions on the coefficients with respect to arbitrarily many prime numbers. These irreducibility criteria were obtained as special cases of the following main result.
\begin{theorem}[Bonciocat \cite{Bonciocat2015}]\label{Bonciocat2015}
Let $f\in \mathbb{Z}[x]$ with $\deg f=n$, let $k\geq 2$, and let $p_1,\ldots,p_k$ be pairwise distinct prime numbers. For $i=1,\ldots,k$, let $w_{i,1},\ldots,w_{i,n_i}$ be the widths of all the segments of the Newton polygon of $f$ with respect to $p_i$ and $\mathscr{S}_{p_i}$ the set of all integers in the interval $(0,\lfloor\frac{n}{2}\rfloor]$ that may be written as a linear combination of $w_{i,1},\ldots,w_{i,n_i}$ with coefficients 0 or 1. If $\mathscr{S}_{p_1}\cap \cdots\cap \mathscr{S}_{p_k}=\emptyset$, then $f$ is irreducible over $\mathbb{Q}$.
\end{theorem}
Theorem \ref{Bonciocat2015} is a marvelous masterpiece in view of its far reaching implications for deriving new irreducibility criteria of Sch\"onemann-Eisenstein-Dumas type. Here, we mention some irreducibility criteria that follow from Theorem  \ref{Bonciocat2015}.
When the Newton polygon of $f$ with respect to each prime $p_i$, $1\leq i\leq k$ consists of single edge, then one has the following irreducibility criterion.
\begin{theorem}[Bonciocat \cite{Bonciocat2015}]\label{Bonciocat2015a}
Let $f=a_0+a_1x+\cdots+a_nx^n\in \mathbb{Z}[x]$ with $a_0a_n\neq 0$; let $k\geq 2$, and let $p_1,\ldots,p_k$ be pairwise distinct prime numbers. Assume that for  each $i=1,\ldots,k$,
\begin{eqnarray*}
v_{p_i}(a_j)\geq \frac{n-j}{n} v_{p_i}(a_0)+\frac{j}{n}v_{p_i}(a_n),~j=1,\ldots,n-1,
\end{eqnarray*}
where exactly one of the integers $v_{p_i}(a_0)$ and $v_{p_i}(a_n)$ is 0, and the nonzero one is denoted by $\alpha_i$. If $\gcd(\gcd(\alpha_1,n),\ldots,\gcd(\alpha_k,n))=1$, then $f$ is irreducible over $\mathbb{Q}$.
\end{theorem}
In the following results we take $k=2$.

If Newton polygon of $f$ with respect to $p_1$ consists of two segments whose slopes have opposite signs, and Newton polygon of $f$ with respect to $p_2$ consists of single edge having positive slope, then one has the following irreducibility criterion.
\begin{theorem}[Bonciocat \cite{Bonciocat2015}]\label{Bonciocat2015b}
Let $f=a_0+a_1x+\cdots+a_nx^n\in \mathbb{Z}[x]$ with $a_0a_n\neq 0$. Let there be two distinct primes $p_1$ and $p_2$ and an index $j\in \{1,\ldots,n-1\}$  such that $j$ is not a multiple of $n/\gcd(v_{p_2}(a_n),n)$, and
\begin{enumerate}[label=$(\roman*)$]
\item $v_{p_1}(a_0)>0$, $\frac{v_{p_1}(a_i)}{j-i}>\frac{v_{p_1}(a_0))}{j}$ for $0<i<j$, $v_{p_1}(a_j)=0$,\\
 $\frac{v_{p_1}(a_i)}{i-j}>\frac{v_{p_1}(a_n))}{n-j}$ for $j<i<n$, $v_{p_1}(a_n)>0$,
\item $\gcd(v_{p_1}(a_0),j)=\gcd(v_{p_1}(a_n),n-j)=1$,
\item $v_{p_2}(a_0)=0$, $\frac{v_{p_2}(a_i)}{i}\geq \frac{v_{p_2}(a_n)}{n}$ for $i\geq 1$, and $v_{p_2}(a_n)>0$.
\end{enumerate}
Then $f$ is irreducible over $\mathbb{Q}$.
\end{theorem}
We now turn to the case when Newton polygon of $f$ with respect to $p_1$ consists of one edge with positive slope, and Newton polygon of $f$ with respect to $p_2$ consists of two edges one of which lies on $x$-axis and the other has positive slope. In view of this, one obtains the following irreducibility criterion.
\begin{theorem}[Bonciocat \cite{Bonciocat2015}]\label{Bonciocat2015c}
Let $f=a_0+a_1x+\cdots+a_nx^n\in \mathbb{Z}[x]$ with $a_0a_n\neq 0$. Let there be two distinct primes $p_1$ and $p_2$ and an index $j<n/\gcd(v_{p_2}(a_n),n)$  such that
\begin{enumerate}[label=$(\roman*)$]
\item $v_{p_1}(a_i)=0$ for $i\leq j$, $\frac{v_{p_1}(a_i)}{i-j}>\frac{v_{p_1}(a_n))}{n-j}$ for $j<i<n$, $v_{p_1}(a_n)>0$,
\item $\gcd(v_{p_1}(a_n),n-j)=1$,
\item $v_{p_2}(a_0)=0$, $\frac{v_{p_2}(a_i)}{i}\geq \frac{v_{p_2}(a_n)}{n}$ for $i\geq 1$, and $v_{p_2}(a_n)>0$.
\end{enumerate}
Then $f$ is irreducible over $\mathbb{Q}$.
\end{theorem}
If Newton polygon of $f$ with respect to $p_1$ consists of one edge with positive slope, and Newton polygon of $f$ with respect to $p_2$ consists of two edges one of which lies on $x$-axis and the other has negative slope, then one has the following result.
\begin{theorem}[Bonciocat \cite{Bonciocat2015}]\label{Bonciocat2015d}
Let $f=a_0+a_1x+\cdots+a_nx^n\in \mathbb{Z}[x]$ with $a_0a_n\neq 0$. Let there be two distinct primes $p_1$ and $p_2$ and an index $j>n-n/\gcd(v_{p_2}(a_n),n)$  such that
\begin{enumerate}[label=$(\roman*)$]
\item $v_{p_1}(a_0)>0$, $\frac{v_{p_1}(a_i)}{j-i}>\frac{v_{p_1}(a_0))}{j}$ for $0<i<j$, $v_{p_1}(a_i)=0$ for $i\geq j$,
\item $\gcd(v_{p_1}(a_0),j)=1$,
\item $v_{p_2}(a_0)=0$, $\frac{v_{p_2}(a_i)}{i}\geq \frac{v_{p_2}(a_n)}{n}$ for $i\geq 1$, and $v_{p_2}(a_n)>0$.
\end{enumerate}
Then $f$ is irreducible over $\mathbb{Q}$.
\end{theorem}
On these lines, there are eight more combinations of the slope-wise distribution of Newton polygons of $f$ with respect to $p_1$ and $p_2$, which accordingly yield different elegant irreducibility criteria, and the reader is referred to  \cite{Bonciocat2015} for the details.
\begin{examples}The irreducibility of the following polynomials is immediate from Theorems \ref{Bonciocat2015a}-\ref{Bonciocat2015d}, respectively.
\begin{enumerate}
\item $8+72(x+x^2+x^3+x^4+x^5)+9x^6$.
\item $3+12(x+x^2+x^3)+4x^4+12(x^5+x^6)$.
\item $1+9(x+x^2)+18(x^3+x^4+x^5+x^6)$.
\item $2+18(x+x^2+x^3)+9(x^4+x^5+x^6)$.
\end{enumerate}
\end{examples}
\section{Factorization results based on Newton index}
In \cite{St}, \c{S}tef\u{a}nescu  proved an opulent factorization result and investigated some factorization properties for univariate polynomials over discrete valuation domains devising unconventionally the properties of Newton index. These will be discussed in the sequel.

Let $f=a_0+a_1x+\cdots+a_nx^n\in R[x]$. For any index $i\in \{0, 1, \ldots, n-1\}$, let $m_i(f)$ denote the slope of the line segment joining the points $(n, v_p(a_n))$ and $(i, v_p(a_i))$. More precisely, we have
\begin{eqnarray*}
m_i(f)=\dfrac{v_p(a_n)-v_p(a_i)}{n-i}.
\end{eqnarray*}
Further, the Newton index $e(f)$ of $f$ is defined as $e(f)=\displaystyle\max_{0\leq i\leq n-1}\{m_i(f)\}$. It follows that
\begin{eqnarray*}
e(gh)=\max\{e(g), e(h)\}, ~\text{for all}~ g, h\in R[x].
\end{eqnarray*}
 In fact, for a discrete valuation domain $(R, v)$ the following splendid result was proved in \cite{St}.
\begin{theorem}[\c{S}tef\u{a}nescu \cite{St}]\label{St}
Let $(R, v)$ be a discrete valuation domain and $f=a_0+a_1x+\cdots+a_nx^n\in R[x]$. Let there be an index $s\in\{0, 1, \ldots, n-1\}$ for which the following conditions are satisfied:
\begin{enumerate}[label=$(\roman*)$]
\item  $v(a_n)=0$,
\item $m_i(f)<m_s(f)$, ~for all~ $i\in\{0, 1, \ldots, n-1\}, i\neq s$,
\item $n(n-s)(m_0(f)-m_s(f))=-1$.
\end{enumerate}
Then $f$ is either irreducible in $R[x]$, or $f$ has a factor whose degree is divisible by $n-s$.
\end{theorem}
A mild generalization of the above result was provided in \cite{S-J-1} to include a wider class of polynomials over such domains.
\begin{theorem}[Kumar and Singh \cite{S-J-1}]\label{St1}
Let $(R, v)$ be a discrete valuation domain and $f(x)=a_0+a_1x+\cdots+a_nx^n\in R[x]$. Let there be an index $s\in\{0, 1, \ldots, n-1\}$ for which the following conditions are satisfied:
\begin{enumerate}[label=$(\roman*)$]
\item $v(a_n)=0,$
\item $m_i(f)<m_s(f)$ for all $i\in\{0, 1, \ldots, n-1\}, i\neq s,$
\item $d_s=\gcd(n-s, v(a_s)$ satisfies:
\begin{eqnarray*}
(-d_s)=\begin{cases}
\displaystyle n(n-s)(m_0(f)-m_s(f)),&~\mbox{if}~s\neq 0,\\
-1,&~\mbox{if}~s=0.
\end{cases}
\end{eqnarray*}
\end{enumerate}
Then any factorization $f(x)=g(x)h(x)$ of $f$ in $R[x]$ has a factor whose degree is a multiple of $(n-s)/d_s$.
\end{theorem}
In view of Theorem \ref{St1}, if we take $s=0$, then $f$ is irreducible.

To the best of our knowledge, in most of the factorization results for polynomials over a discrete valuation domain $(R,v)$, the coprimality of $v(a_s)$ and $n-s$ is crucial, whenever $s$ is the smallest index for which the minimum of the quantity $v(a_i)/(n-i)$, $0\leq i\leq n-1$ is $v(a_s)/(n-s)$. To fill the gap in the case when $v(a_s)$ and $n-s$ are not coprime, the following factorization result was proved in \cite{S-J-1}.
\begin{lemma}[Kumar and Singh \cite{S-J-1}]\label{L-S-J-1}
 If $(R,v)$ is a discrete valuation domain, $f=a_0+a_1x+\cdots+a_nx^n \in R[x]$ is such that  $v(a_n)=0,$ and $d_s=\gcd(v(a_s),n-s)>1$, then any factorization $f(x)=g(x)h(x)$ of $f$ in $R[x]$ has $\max\{\deg g,\deg h\}\geq (n-s)/d_s$.
\end{lemma}
Further, the following factorization result was proved in \cite[Theorem 3]{S-J-1}.
\begin{theorem}[Kumar and Singh \cite{S-J-1}]
Let $(R,v)$ be a discrete valuation domain and $f=a_0+a_1x+\cdots+a_nx^n \in R[x]$ with $v(a_n)=0$. Let there be an index $s\in \{0,\ldots,n-1\}$ such that $m_i(f)<m_s(f)$ for all $i=0,1,\ldots,n-1$, $i\neq s$ and $d_s=\gcd(n-s,v(a_s))$. Then any factorization $f(x)=g(x)h(x)$ of $f$ in ${R}[x]$ has $\max\{\deg{g},\deg h\}\geq (n-s)/d_s$.
\end{theorem}
\begin{examples}
\begin{enumerate}
\item For a prime number $p$, let $v=v_p$ denote the $p$-adic valuation on $\mathbb{Z}$. For an odd positive integer $n\geq 5$, the polynomial
    \begin{eqnarray*}
(5+18x^2)3^{n-2}+3^{n-3}x+10x^n
    \end{eqnarray*}
satisfies the hypothesis of Theorem \ref{St1} for $s=1$ and $d_1=2$ by taking $p=3$. So, the given polynomial is  irreducible, or has a factor whose degree is a multiple of $(n-1)/2$.
\item Let $d>2$ be a positive integer. For $f\in \mathbb{Z}[x]$, let $v(f)=-\deg (f)$ for $f\neq 0$ and $v(0)=\infty$, the degree valuation on $\mathbb{Z}[x]$, the polynomial
    \begin{eqnarray*}
1+x+(2+x^2)y+y^{2d+1}\in \mathbb{Z}[x,y]
    \end{eqnarray*}
in $y$ with coefficients from $\mathbb{Z}[x]$ satisfies the hypothesis of Theorem \ref{St1} for $s=1$ and $d_1=d$. So, the given polynomial is irreducible,  or has a factor whose degree is an even positive integer.
\end{enumerate}
\end{examples}
\section{Irreducibility via primality and location of zeros}
In \cite{Polya}, P\'olya and Szego mentioned a beautiful irreducibility criterion due to A. Cohn, which states that if  a prime number $p$ has decimal expansion $p=a_0+a_1 10^1+\cdots+a_n 10^n$, $0\leq a_i\leq 9$, then the polynomial $a_0+a_1 x+\cdots+a_n x^n$ is irreducible in $\mathbb{Z}[x]$. This result was generalized to arbitrary base $b\geq 2$ by Brillhart et al. \cite{Brillhart}. Bonciocat et al. \cite{Bonciocat4} further generalized this result for the polynomials taking prime power values. More precisely, the following result was proved.
\begin{theorem}[Bonciocat et al. \cite{Bonciocat4}]
For a positive integer $N\geq 2$ and a prime $p$, if $p^N$ is expanded in the number system with base $b\geq 2$ as
\begin{eqnarray*}
p^N=a_0+a_1 b+\cdots+a_n b^n,~0\leq a_1,\ldots,a_n\leq b-1,
\end{eqnarray*}
where $p$ does not divide $\sum_{k=1}^n ka_kb^{k-1}$, then the polynomial $a_0+a_1 x+\cdots+a_n x^n\in\mathbb{Z}[x]$ is irreducible over $\mathbb{Q}$.
\end{theorem}
The irreducibility criterion due to A. Cohn and its generalizations apprise of the fascinating fact that prime numbers bear a close connection with irreducible polynomials. This adherence has been a persuasive point to ponder on as is evident from the classical open problem profoundly known as Buniakowski's conjecture (1854) which states that if $f\in \mathbb Z[x]$ is an irreducible polynomial such that the integers in the set $f (\mathbb N)$ have no factor in common other than 1, then $f$ takes prime values infinitely often. It is immediate that if $f$ takes prime values for infinitely many values of $n$, then it must be irreducible. So the converse of Buniakowski's conjecture holds in the affirmative. In anticipation of a tenacious converse of the Buniakowski's conjecture, Murty in \cite{M} proved an exquisite result and deduced the irreducibility of an arbitrary polynomial $f\in \mathbb Z[x]$ under the hypothesis that $f(m)$ be prime for a sufficiently large integer $m$. This exquisite irreducibility criterion was then generalized by Girstmair in \cite{Gi} for primitive polynomials $f\in \mathbb Z[x]$ together with the hypothesis that $|f(m)| = pd$ for a sufficiently large $m, p\nmid d$. To state these results, let us recall that the height $H_f$ of a polynomial $f=a_0+a_1x+\cdots +a_nx^n\in \mathbb Z[x]$ is defined as
\begin{eqnarray*}
H_f=\max_{0\leq i\leq n-1}\Big\{|a_i|/|a_n|\Big\}.
\end{eqnarray*}
Observe that if $x\in \mathbb{C}$ is such that $|x|\geq 1+H_f$, then $-1/|x|>-1/(1+H_f)$; $-|a_i|/|a_n|\geq (-H_f)$,  and we have
\begin{eqnarray*}
\frac{|f(x)|}{|a_n||x|^n} &\geq &\Bigl{|}1-\sum_{i=0}^{n-1}\frac{|a_i|}{|a_n||x|^{n-i}}\Bigr{|}\geq  \Bigl{|}1-\sum_{i=0}^{n-1}\frac{H_f}{(1+H_f)^{n-i}}\Bigr{|}=\frac{1}{(1+H_f)^{n}}>0,
\end{eqnarray*}
which shows that each zero $\theta$ of $f$ satisfies $|\theta|<1+H_f$. Now we have the following irreducibility criteria.
\begin{theorem}[Murty \cite{M}]\label{M} Let $f=a_0+a_1x+\cdots +a_nx^n\in \mathbb Z[x]$ be such that  there exists an integer $m\geq H_f+2$ for which $f(m)$ is prime. Then $f$ is irreducible.
\end{theorem}
\begin{theorem}[Girstmair \cite{Gi}]\label{G-1}
Let $f=a_0+a_1x+\cdots +a_nx^n\in \mathbb Z[x]$ be a primitive polynomial. If there exist integers $d, m\in \mathbb N$ such that $m\geq H_f+d+1, f(m)=\pm d\cdot p$, where $p$ is a prime with $p\nmid d$, then $f$ is irreducible.
\end{theorem}
Observe that Theorem \ref{M} is a special case of Theorem \ref{G-1} corresponding to $d=1$. 

Recently, in \cite{J-S-2}, the authors obtained the following generalizations of Theorem \ref{G-1}.
\begin{theorem}[Singh and Kumar \cite{J-S-2}]\label{J-S-2a}
Let $f=a_0+a_1 x+\cdots+a_nx^n\in \Bbb{Z}[x]$ be a primitive polynomial. If there exist natural numbers $m$, $d$, $k$, and a prime $p\nmid d$ such that $m\geq H_f+d+1$, $f(m)=\pm p^k d$, and for $k>1$, $p\nmid f'(m)$, then $f$ is irreducible in $\Bbb{Z}[x]$.
\end{theorem}
\begin{theorem}[Singh and Kumar \cite{J-S-2}]\label{J-S-2b}
Let $f=a_0+a_1 x+\cdots+a_nx^n\in \Bbb{Z}[x]$ be a primitive polynomial. If there exist natural numbers $m$, $d$, $k$, $j\leq n$, and a prime $p\nmid d$ such that $m\geq H_f+d+1$, $f(m)=\pm p^k d$, $\gcd(k,j)=1$, $p^k \mid \frac{f^{(i)}(m)}{i!}$ for each index $i=0,\ldots,j-1$, and for $k>1$, $p\nmid \frac{f^{(j)}(m)}{j!}$, then $f$ is irreducible in $\Bbb{Z}[x]$.
\end{theorem}
\begin{examples}
\begin{enumerate}
\item The following example (see \cite{J-S-2}) shows a comparison between Theorem \ref{J-S-2a} and  Theorem \ref{G-1}.  The polynomial
\begin{eqnarray*}
f_1 &=& 7+5x-16x^2+6x^3+2x^4+7x^5+x^6+6x^7+2x^8+8x^9+4x^{10}
\end{eqnarray*}
satisfies the hypothesis of Theorem \ref{J-S-2a} for $H_{f_1}=4$, $f_1(10)=137^5$, $d=1$, $137\nmid f_1'(10)$, and so, the polynomial $f_1$ is irreducible in $\mathbb{Z}[x]$. The smallest value of $n$ for which $f_1(n)$ satisfies Theorem \ref{G-1} is $n=50$, where $f_1(50)=406332830325710257$, a large prime number.
\item If we consider the polynomial
\begin{eqnarray*}
f_2 &=& 49-14x+x^2+49x^7
\end{eqnarray*}
in $\mathbb{Z}[x]$, then $f_2$ satisfies the hypothesis of Theorem \ref{J-S-2b} for $p=7$, $d=1$, since  $f_2(7)=7^9=f_2'(7)$, $f_2''(7)/2=1+3\times 7^{8}\equiv1\mod 7$, and $H_{f_2}=1$ so that $1+d+H_{f_2}=3<7$. So, the polynomial $f_2$ is irreducible in $\mathbb{Z}[x]$.
\end{enumerate}
\end{examples}
The most alluring part of Theorems \ref{J-S-2a}-\ref{J-S-2b} is that the proofs rest on the following fundamental lemma of Singh and Kumar
\cite[Lemma 3]{J-S-3}.
\begin{lemma}[Singh and Kumar \cite{J-S-3}]\label{J-S-3}
Let $f=a_0+ a_{1}x+\cdots+a_n x^n$, $f_1=b_0+b_1x+\cdots+b_rx^r$, and  $f_2=c_0+c_1x+\cdots+c_{n-r}x^{n-r}$ be nonconstant polynomials in $\Bbb{Z}[x]$ such that $f(x)=f_1(x)f_2(x)$. Suppose that there is a prime number $p$ and positive integers $k\geq 2$ and $j\leq n$ such that $p^k\mid \gcd(a_0,a_1,\ldots,a_{j-1})$, $p^{k+1}\nmid a_0$, and $\gcd(k,j)=1$. If $p\mid b_0$ and $p\mid c_0$, then $p\mid a_j$.
\end{lemma}
In the context, another appealing feature is the vociferous coherence with the notion of location of the zeros of given polynomial being tested for irreducibility \cite{P}. Further, Lemma \ref{J-S-3} leads to the efficacious construction of two major irreducibility certificates proved using elementary divisibility properties of integers in the same article which may be of independent interest as well. These results are as follows:
\begin{theorem}[Singh and Kumar \cite{J-S-3}]\label{J-S-3a}
   Let $f=a_0+ a_{1}x+\cdots+a_n x^n\in \Bbb{Z}[x]$ be a primitive polynomial such that each zero $\theta$ of $f$ satisfies $|\theta|>d$, where $a_0=\pm p^k d$ for some positive integers $k$ and $d$, and a prime $p\nmid d$. If $j\in\{1,\ldots,n\}$ is such that $\gcd(k,j)=1$,  $p^k\mid \gcd(a_0,a_1,\ldots,a_{j-1})$ and for $k>1$,  $p\nmid a_{j}$, then $f$ is irreducible in $\Bbb{Z}[x]$.
\end{theorem}
\begin{theorem}[Singh and Kumar \cite{J-S-3}]\label{J-S-3b}
    Let $f=a_0+ a_{1}x+\cdots+a_n x^n\in \Bbb{Z}[x]$ be a primitive polynomial such that each zero $\theta$ of $f$ satisfies $|\theta|>d$, where $a_n=\pm p^k d$ for some positive integers $k$ and  $d$, and a prime $p\nmid d$. Let $j\in \{1,\ldots, n\}$ be such that $\gcd(k,j)=1$, $p^k\mid \gcd(a_{n-j+1},a_{n-j+2},\ldots,a_{n})$ and for $k>1$,  $p\nmid a_{n-j}$. If $|a_0/q|\leq |a_n|$, where $q$ is the smallest prime divisor of $a_0$, then $f$ is irreducible in $\Bbb{Z}[x]$.
\end{theorem}
Independent proof of Theorem \ref{J-S-3a} (and hence of Theorem \ref{J-S-3b}) for small values of $j$ are direct and easy to comprehend unlike that in the general version, where the criterion was established by directly comparing the coefficients of the given polynomial. For acquaintance, we include the proofs of Theorem \ref{J-S-3a} in the case when $j=1,2,3$.
\begin{theorem}\label{th:1}
Let $f=a_0+ a_{1}x+\cdots+a_n x^n\in \Bbb{Z}[x]$ be a primitive polynomial such that each zero $\theta$ of $f$ satisfies $|\theta|>d$, where $a_0=\pm p^k d$ for some  positive integers $k$, $d$, and prime $p\nmid a_1d$.  Then $f$ is irreducible in $\Bbb{Z}[x]$.
\end{theorem}
\begin{proof}
If possible, let $f(x)=g(x)h(x)$, where
    $g=b_0+b_1x+\cdots+b_mx^m$ and $h=c_0+c_1x+\cdots+c_{n-m}x^{n-m}$ are nonconstant polynomials in $\Bbb{Z}[x]$. Then  $b_0c_0=\pm p^kd$, $b_0c_1+b_1c_0=a_1$, and $b_mc_{n-m}=a_n$.
    Since each zero $\theta$ of $f$ satisfies $|\theta|>d\geq 1$, we must have $|b_0/b_m|>d$ and $|c_0/c_{n-m}|>d$ so that $|b_0|>d$ and $|c_0|>d$. If $k=1$, then $|b_0||c_0|=pd$, and so, $p$ divides one of $|b_0|$ or $|c_0|$. Consequently, either $|b_0|\leq d$ or $|c_0|\leq d$. This  contradicts the conclusion of the preceding sentence. Now assume that $k>1$. Since $|b_0||c_0|=p^kd$ with $|b_0|>d$ and $|c_0|>d$, it follows that $p\mid b_0$ and $p\mid c_0$. Then $p\mid (b_0c_1+b_1c_0)=a_1$,  which contradicts the hypothesis.
\end{proof}
\begin{theorem}\label{th:2}
    Let $f=a_0+ a_{1}x+\cdots+a_n x^n\in \Bbb{Z}[x]$ be a primitive polynomial such that each zero $\theta$ of $f$ satisfies $|\theta|>d$, where $a_0=\pm p^kd$ and  $\gcd(a_0,a_1)=p^k$ for some positive integers $k$, $d$, and prime $p\nmid a_2 d$ with $2\nmid k$. Then $f$ is irreducible in $\Bbb{Z}[x]$.
\end{theorem}
\begin{proof}
Assume that $f(x)=g(x)h(x)$, where
    $g=b_0+b_1x+\cdots+b_mx^m$ and $h=c_0+c_1x+\cdots+c_{n-m}x^{n-m}$ are nonconstant polynomials in $\Bbb{Z}[x]$. Then $b_0c_0=\pm p^{k}d$, $b_0c_1+b_1c_0=a_1$, $b_0c_2+b_1c_1+b_2c_0=a_2$, and $b_mc_{n-m}=a_n$. Since each zero $\theta$ of $f$ satisfies $|\theta|>d$, we must have $|b_0/b_m|>d$ and $|c_0/c_{n-m}|>d$ so that $|b_0|>|b_m|$ and $|c_0|>|c_{n-m}|$.
We will assume that $k>1$ since the proof for the case $k=1$ is same as that in the proof of Theorem \ref{th:1}. So, we have $|b_0||c_0|=|a_0|=p^{k}d$. Consequently,  there exists a positive integer $\ell\leq k-\ell$,  such that  $p^\ell$ divides $|b_0|$ and $p^{k-\ell}$ divides $|c_0|$. Since $2\nmid k$, we must have $\ell<k-\ell$. Since $b_0c_1+b_1c_0=a_1$, it follows that $p^{k-2\ell}$ divides $c_1$. Consequently, $p\mid (b_0c_2+b_1c_1+b_2c_0)=a_2$, which contradicts the hypothesis.
\end{proof}
\begin{theorem}\label{th:3}
    Let $f=a_0+ a_{1}x+\cdots+a_n x^n\in \Bbb{Z}[x]$ be a primitive polynomial such that each zero $\theta$ of $f$ satisfies $|\theta|>d$, where  $a_0=\pm p^kd$ and $\gcd(a_0,a_1,a_2)=p^k$ for  some positive integers $k$, $d$, and prime  $p\nmid a_3d$ and $\gcd(k,3)=1$. Then $f$ is irreducible in $\Bbb{Z}[x]$.
\end{theorem}
\begin{proof}
Assume $f(x)=g(x)h(x)$, where
    $g=b_0+b_1x+\cdots+b_mx^m$ and $h=c_0+c_1x+\cdots+c_{n-m}x^{n-m}$ are nonconstant polynomials in $\Bbb{Z}[x]$. As in the proof of Theorem \ref{th:2}, the condition that each zero $\theta$ of $f$ satisfies $|\theta|>d$ implies
$|b_0/b_m|>d$ and $|c_0/c_{n-m}|>d$ so that $|b_0|>d$ and $|c_0|>d$.
We will assume that $k>1$ since the proof for the case $k=1$ is as before.
Then $|b_0||c_0|=|a_0|=p^{k}d$, and so, there exists a positive integer $\ell\leq k-\ell$ such that  $p^\ell$ divides $|b_0|$ and $p^{k-\ell}$ divides $|c_0|$.

First assume that $\ell<k-\ell$. This in view of the fact that $b_0c_1+b_1c_0=a_1$ shows that $p^{k-2\ell}$ divides $|c_1|$. Since $\gcd(k,3)=1$, we have $k\neq 3\ell$. If $\ell<k-2\ell$, then $p^{k-3\ell}$ divides $c_2$, since $p^k\mid a_2=b_0c_2+b_1c_1+b_2c_0$. On the other hand if $\ell>k-2\ell$, then $b_0c_2+b_1c_1+b_2c_0=a_2$ implies that $p^{\min\{3\ell-k,\ell\}}$ divides $b_1$,  if $p\nmid c_1p^{-k+2\ell}$. If $p\mid c_1p^{-k+2\ell}$, then as $b_0c_1+b_1c_0=a_1$, we have that $p\mid b_1$. Thus, $\ell<k-\ell$ implies that  $p\mid b_1c_2$. But then $p\mid (b_0c_3+b_1c_2+b_2c_1+b_3c_0)=a_3$, which contradicts the hypothesis.

Now assume that  $\ell=k-\ell$. In this case $p^{\ell}\mid(c_1\pm b_1)$. Further,  $p^\ell\mid (a_2-b_0c_2-b_2c_0)=b_1c_1$. Consequently, $p\mid b_1$ and $p\mid c_1$.  We then have $p\mid (b_0c_3+b_1c_2+b_2c_1+b_3c_0)=a_3$, which contradicts the hypothesis.
\end{proof}
\begin{examples} The irreducibility of the following two polynomials is immediate from Theorems \ref{J-S-3a} and \ref{J-S-3b}, respectively for a prime $p$.
\begin{enumerate}
\item $p^{a+1}(1+x+x^2+\cdots+x^{a-1})+(p^a-1) x^{a}+p^{a-1}x^{a+1}(1+x+\cdots+x^{n-a-1}),~1\leq a<n-1$ for $k=a+1$ and $j=a$.
\item $p^{n}(n+x+x^2+\cdots+x^{n-j-1})+m x^{n-j}+p^n(x^{n-j+1}+\cdots+x^{n})$ for each $j=n-1, n-2,\ldots, 1$ and  $m=1,\ldots,p-1$.
\end{enumerate}
\end{examples}
Quite recently in \cite{Z2023}, Zhang and Yuan have provided a short proof of Theorem \ref{J-S-3a} using Newton polygon approach. In the same paper, the authors note that in particular,  Theorem \ref{J-S-3a} implies the following conjecture of Koley and Reddy.
\begin{conjecture}[Koley and Reddy \cite{Koley}]\label{Koleya}
Let $f=a_0+a_1x+\cdots+a_nx^n\in\mathbb{Z}[x]$  be such that  $a_0=\pm p^k$ for some prime $p$ and positive integer $k$, and $a_1=a_2=\cdots=a_{q-1}=0$ for some prime $q\leq n$. If $p\nmid a_qa_n$, $q\nmid k$, and $p^k>|a_q|+|a_{q+1}|+\cdots+|a_n|$, then the polynomial $f$ is irreducible over $\mathbb{Q}$.
\end{conjecture}
\begin{proof}
Since $|a_0|=p^k$ and $p\nmid a_q a_n$, it follows that $f$ is primitive. If $|x|\leq 1$, then we have
\begin{equation*}
    \begin{split}
|f(x)|&\geq p^k-|a_q||x|-|a_{q+1}||x|^2-\cdots-|a_n||x|^n\\&\geq p^k-|a_q|-|a_{q+1}|-\cdots-|a_n|>0,
    \end{split}
\end{equation*}
which shows that $f(x)\neq 0$ for $|x|\leq 1$. This shows that each zero $\theta$ of $f$ satisfies $|\theta|>1$. Observe that the polynomial $f$ satisfies the hypothesis of Theorem \ref{J-S-3a} for $j=q$ and $d=1$, and so  $f$ is irreducible in $\mathbb{Z}[x]$.
\end{proof}
Further, using the Newton polygon technique, the authors in  \cite{Z2023} also provide an alternative proof of the following Conjecture of Harrington.
\begin{conjecture}[Harrington \cite{Harrington}]\label{Harrington}
If  $f=\pm a(1+x^{n-1})+x^n\in \mathbb{Z}[x]$ with $a\geq 2$, then $f$ is irreducible unless $f=x^2+4x+4$.
\end{conjecture}
In the Conjecture \ref{Koleya}, the constant term was taken to be a prime exponent. However,   in \cite{Bonciocat2}, the following irreducibility criterion was proved for the polynomials having integer coefficients  in which the leading coefficient is divisible by a large prime power.
\begin{theorem}[Bonciocat et al. \cite{Bonciocat2}]\label{Bonciocate2}
If $f = p^ma_nx^n+a_{n-1}p^sx^{n-1}+a_{n-2}x^{n-2}+\cdots+a_1 x+a_0\in \Bbb{Z}[x]$ with $n\geq 2$, $m\geq 1$, $s\geq 0$, $a_0a_{n-1}a_n\neq 0$, $p$  a prime number with $p\nmid a_{n-1}a_{n}$, and if
\begin{eqnarray*}
p^m &>& |a_{n-1}|p^{2s}+\sum_{i=2}^{n}|a_n^{i-1} a_{n-i}|p^{is},
\end{eqnarray*}
then $f$ is irreducible over $\mathbb{Q}$.
\end{theorem}
Note that if we let $\delta=p^s|a_n|$, then we observe that the preceding inequality is equivalent to
 \begin{eqnarray*}
\frac{p^m|a_n|}{\delta^n} &>& \frac{|a_{n-1}|p^{s}}{\delta^{n-1}}+\sum_{i=2}^{n}\frac{|a_{n-i}|}{\delta^{n-i}},
\end{eqnarray*}
which shows that each zero $\theta$ of $f$ satisfies $|\theta|<({1}/{\delta})\leq 1$.

Another superb ethereal irreducibility criterion for polynomials having integer coefficients appeared in \cite{B}, wherein the coefficients of the given polynomial were monotonically decreasing with the constant term being a prime and some splendid equivalent statements were proved. These assertions are stated as follows.
\begin{theorem}[Bevelacqua \cite{B}]\label{Bevelacqua}
Let $f = a_0 + a_1x +\cdots+a_n x^n\in \Bbb{Z}[x]$ be such that $a_0\geq a_1\geq \cdots \geq a_n > 0$. If $a_0$ is prime, then the following statements are equivalent:
\begin{enumerate}[label=$(\roman*)$]
\item $f$ is irreducible in $\Bbb{Z}[x]$,
\item for any $s\geq 1$, $f(x^s)=a_0 + a_1x^s +\cdots+a_n x^{sn}$ is irreducible in $\Bbb{Z}[x]$,
\item the list $(a_0,a_1,\ldots,a_n)$ does not consist of $(n+1)/d$ consecutive constant lists of length $d>1$.
\end{enumerate}
\end{theorem}
In \cite{J-S-4}, the authors proposed some succinct conditions on the coefficients of $f$ to obtain a simple yet interesting generalization of Theorem \ref{Bevelacqua} and this yielded a wider class of irreducible polynomials.
\begin{theorem}[Singh and Kumar \cite{J-S-4}]\label{J-S-4a}
Let $f = a_0 + a_1x +\cdots+a_n x^n\in \Bbb{Z}[x]$ be such that $a_0\geq a_1\geq \cdots \geq a_n > 0$. If either $a_0$ is prime, or $a_n$ is prime and $a_n\geq a_0/q$, where $q$ denotes the smallest prime divisor of $a_0$, then the following statements are equivalent:
\begin{enumerate}[label=$(\roman*)$]
\item $f$ is irreducible in $\Bbb{Z}[x]$.
\item the polynomial $g$ such that for any $s\geq 1$, $g(x)=f(x^s)$ is irreducible in $\Bbb{Z}[x]$.
\item the list $a_0,a_1,\ldots,a_n$ does not consist of $(n+1)/d$ consecutive constant lists of length $d>1$.
\end{enumerate}
\end{theorem}
In  Bevelacqua's irreducibility criterion (Theorem \ref{Bevelacqua}), all the zeros of the given polynomial lie outside the closed unit disc in the complex plane. However in \cite{J-S-4}, the authors obtained another extension of Theorem \ref{Bevelacqua} for a class of polynomials in which all zeros of the given polynomial lie outside the closed disc  $|z|\leq \lambda$   in the complex plane for some $\lambda\in (0,1)$.
The result is precisely stated as follows.
\begin{theorem}[Singh and Kumar \cite{J-S-4}]\label{J-S-4}
Let $f=a_0+a_1x+\cdots+a_n x^n\in \mathbb{Z}[x]$ be such that either $a_0$ or $a_n$ is prime. If there exists $\lambda\in (0,1)$ such that
\begin{enumerate}[label=$(\roman*)$]
\item[(a)] $a_0\geq \lambda a_1\geq \lambda^2a_2\geq \cdots \geq \lambda^n a_n>0$,
\item[(b)] $a_0\leq a_n \lambda ^{n-1}$ and $1\leq q \lambda^{n-1}$, where $q$ is the smallest prime divisor of $a_n$, then the following statements are equivalent:
\item $f$ is irreducible in $\Bbb{Z}[x]$.
\item the polynomial $g$ such that for any $s\geq 1$, $g(x)=f(x^s)$ is irreducible in $\Bbb{Z}[x]$.
\item the list $a_0,\lambda a_1,\ldots,\lambda^n a_n$ does not consist of $(n+1)/d$ consecutive constant lists of length $d>1$.
\end{enumerate}
\end{theorem}
\begin{examples}
\begin{enumerate}
\item For a prime  $p$ and any positive integer $n$, the list  of the coefficients of the polynomial
\begin{eqnarray*}
p+(p-1)(x+x^2+\cdots+x^{n-1}+x^n)
\end{eqnarray*}
is
\begin{eqnarray*}
(p,\underbrace{p-1,p-1,\ldots,p-1}_{n-\text{times}}),
\end{eqnarray*}
which consists of two constant lists of lengths $1$ and $n$, respectively, and so, $d=\gcd(n,1)=1$. By Theorem \ref{Bevelacqua}, the given polynomial is irreducible in $\mathbb{Z}[x]$.
\item The irreducibility of the polynomial
\begin{eqnarray*}
10+7(x+x^2+\cdots+x^{n-1}+x^n)
\end{eqnarray*}
is immediate from Theorem \ref{J-S-4a} with $a_0=10$, $a_n=7$, and $q=2$.
\item Observe that the polynomial
\begin{eqnarray*}
59+67x+75x^2+85x^3+96x^4+100x^5
\end{eqnarray*}
satisfies the hypothesis of Theorem \ref{J-S-4} for
\begin{eqnarray*}
\lambda=\sqrt[4]{{59}/{100}}\approx 0.876421,~a_0=59,~n=5,~a_5=100,~q=2,
\end{eqnarray*}
since here we have
\begin{eqnarray*}
(59,67\lambda, 75\lambda^2,85\lambda^3,96\lambda^4,100\lambda^5)\approx(59,58.72,57.22,56.64,51.7),
\end{eqnarray*}
and so, $59>67\lambda> 75\lambda^2>85\lambda^3>96\lambda^4>100\lambda^5$ with $59=100 \lambda^4$, and $1<1.18\approx q\lambda^{4}$. Thus the given polynomial is irreducible in $\mathbb{Z}[x]$.
\end{enumerate}
\end{examples}
Now we discuss some interesting results recently obtained by Bonciocat et al. \cite{Bonciocat2021} for irreducibility of the polynomials having integer coefficients whose zeros lie within an Apollonius circle.  Recall that the locus of the point which moves in the Euclidean plane $\mathbb{R}^2$ in such a way that the ratio of its distances from two fixed points is constant ($\neq 1$) is a circle called Apollonius circle. If $A(\alpha_1,0)$ and $B(\alpha_2,0)$ are the two fixed points and $P(x,y)$ is the point in the plane such that $PB=k\times PA$ for some constant $1\neq k>0$, then the equation of Apollonius circle $\mbox{Ap}(\alpha_1,\alpha_2,k)$ defined by $\alpha_1,\alpha_2,k$ is
\begin{eqnarray*}
\Bigl(x-\alpha_1+\frac{\alpha_2-\alpha_1}{k^2-1}\Bigr)^2+y^2=k^2\Bigl(\frac{\alpha_2-\alpha_1}{k^2-1}\Bigr).
\end{eqnarray*}
For a nonconstant polynomial $f$ having integer coefficients, let  $f(a)\neq 0$ for some integer $a$. A divisor $d$ of $f(a)$ is said to be an admissible divisor of $f(a)$ if $\gcd(d,f(a)/d)$ divides $\gcd(f(a),f'(a))$. Now we have the following result.
\begin{theorem}[Bonciocat et al. \cite{Bonciocat2021}]\label{Bonciocat2021}
Let $f=a_0+a_1x+\cdots+a_nx^n \in\mathbb{Z}[x]$ be such that there exist two integers $\alpha_1$ and $\alpha_2$ for which $0<|f(\alpha_1)|<|f(\alpha_2)|$. Let
\begin{eqnarray*}
q=\max_{d_i\in \mathcal{D}(f(\alpha_i)),~i=1,2}\Bigl\{\frac{d_2}{d_1}\leq \sqrt{\frac{|f(\alpha_2)|}{|f(\alpha_1)|}}\Bigr\},
\end{eqnarray*}
where $\mathcal{D}(f(\alpha_i))$ is the set of all admissible divisors of $f(\alpha_i)$. Then the following assertions hold.
\begin{enumerate}[label=$(\roman*)$]
\item If $q>1$ and all the zeros of $f$ lie inside the Apollonius circle $\mbox{Ap}(\alpha_1,\alpha_2,q)$, then $f$ is irreducible over $\mathbb{Q}$.
\item If $q>1$ and all the zeros of $f$ lie inside the Apollonius circle $\mbox{Ap}(\alpha_1,\alpha_2,\sqrt{q})$, and if $f$ has no rational zero, then $f$ is irreducible over $\mathbb{Q}$.
\item If $q=1$, $\alpha_1<\alpha_2$ and all the zeros of $f$ lie in the half plane $x<(\alpha_1+\alpha_2)/2$, or if $\alpha_2<\alpha_1$ and all zeros of $f$ lie in the half plane $x>(\alpha_1+\alpha_2)/2$, then $f$ is irreducible over $\mathbb{Q}$.
\end{enumerate}
\end{theorem}
Several irreducibility criteria have been stated for the polynomials having integer coefficients whose all zeros lie either within some open disk or all zeros lie outside a closed disk. Let us concentrate on some irreducibility  criteria for the polynomials whose zeros lie outside a closed annulus in the complex plane. In this regards, one of the oldest irreducibility criteria is due to Perron \cite{P}, which states that if $f=a_0+a_1x+\cdots+a_nx^n\in \mathbb{Z}[x]$ is such that $a_0\neq 0$, $a_n=1$,  $n\geq 2$, and
\begin{eqnarray*}
|a_{n-1}| &>& 1+|a_{n-2}|+\cdots+|a_1|+|a_0|,
\end{eqnarray*}
then $f$ is irreducible. The given condition on the coefficients of $f$ tells us that exactly $n-1$ zeros of  $f$ lie within the disk $|x|<1$, and the remaining one zero lies outside the closed disk $|x|\leq 1$. This fact can be easily proved using Rouch\'e's Theorem. Indeed if we take $g(x)=a_{n-1}x^{n-1}$, then $g$ has $n-1$ zeros in the interior of the unit circle  $|x|=1$ in the complex plane.
On the unit circle $|x|=1$, we have
\begin{eqnarray*}
|g(x)|_{|x|=1}=|a_{n-1}|> 1+|a_{n-2}|+\cdots+|a_1|+|a_0|\geq |f(x)-g(x)|_{|x|=1}.
\end{eqnarray*}
Consequently, by Rouch\'e's Theorem, the polynomials $g$ and $f$ have same number of zeros in the unit disk $|x|<1$, that is,  exactly $n-1$ zeros of $f$ lie in the unit disk $|x|<1$. Since $f$ has $n$ zeros, the remaining one zero of $f$ lies outside the unit disk $|x|<1$. Now the irreducibility of $f$ is immediate.

There have been several generalizations and extensions of the classical irreducibility criterion due to Perron in the past. In \cite{Bonciocat4}, Bonciocat et al. proved the following key result.
\begin{lemma}[Bonciocat et al. \cite{Bonciocat4}]\label{Bonciocat4a}
Let $f\in \mathbb{Z}[x]$ be such that $f(m)=p^Nd$ for some integers $m,N,d$, and a prime $p$ with $N\geq 2$ and $p\nmid df'(m)$. If there exist real numbers $\alpha$ and $\beta$ for which $\alpha<|m|-|d|<|m|+|d|<\beta$, and $f$ has no zero within the annulus $\alpha<|z|<\beta$, then $f$ is irreducible over $\mathbb{Q}$.
\end{lemma}
An ingenious use of Lemma \ref{Bonciocat4a} yields the following elegant irreducibility criteria.
\begin{theorem}[Bonciocat et al. \cite{Bonciocat4}]\label{Bonciocat4b}
Let $f=a_0+a_1x+\cdots+a_nx^n\in \mathbb{Z}[x]$ be such that $a_0a_n\neq 0$. Suppose there exist integers $m$, $N$, $d$  with $N\geq 2$ and a prime $p$ for which $f(m)=p^N d$, $p\nmid f'(m)d$ and
\begin{eqnarray*}
|a_0|>|a_1|\beta+|a_2|\beta^2 +\cdots+|a_n| \beta^n,
\end{eqnarray*}
where $\beta=|m|+|d|$. Then $f$ is irreducible over $\mathbb{Q}$.
\end{theorem}
\begin{theorem}[Bonciocat et al. \cite{Bonciocat4}]\label{Bonciocat4c}
Let $f=a_0+a_1x+\cdots+a_nx^n\in \mathbb{Z}[x]$ be such that $a_0a_n\neq 0$. Suppose there exist integers $m$, $N$, $d$  with $|m|>|d|$, $N\geq 2$ and a prime $p$ for which $f(m)=p^N d$, $p\nmid f'(m)d$ and
\begin{eqnarray*}
|a_j|>\beta^{n-j}\sum_{k\neq j} |a_k|,
\end{eqnarray*}
where $\beta=|m|+|d|$. Then $f$ is irreducible over $\mathbb{Q}$.
\end{theorem}
\begin{theorem}[Bonciocat et al. \cite{Bonciocat4}]\label{Bonciocat4d}
Let $f=a_0+a_1x+\cdots+a_nx^n\in \mathbb{Z}[x]$ be such that $a_0a_n\neq 0$. Suppose there exist integers $m$, $N$, $d$  with $|m|>|d|$, $N\geq 2$ and a prime $p$ for which $f(m)=p^N d$, $p\nmid f'(m)d$ and
\begin{eqnarray*}
|a_n|\alpha^n>|a_0|+|a_1|\alpha +\cdots+|a_{n-1}| \alpha^{n-1},
\end{eqnarray*}
where $\alpha=|m|-|d|$. Then $f$ is irreducible over $\mathbb{Q}$.
\end{theorem}
We observe that the conclusions of Theorems \ref{Bonciocat4b}-\ref{Bonciocat4c} still hold if the condition $\beta=|m|+|d|$ is replaced by any real number $\beta$ satisfying $\beta\geq |m|+|d|$. Similarly, the conclusion of Theorem \ref{Bonciocat4d} still holds even if we replace the condition $\alpha=|m|-|d|$, by a real number $\alpha$ satisfying $\alpha\leq |m|-|d|$.

Recently in \cite{J-R-1}, Singh and Garg obtained a factorization result for polynomials having integer coefficients  satisfying a condition analogous to that of Perron. The result is  stated as follows.
\begin{theorem}[Singh and Garg \cite{J-R-1}]\label{J-R-1a}
Let $f = a_0 + a_1x +\cdots+a_n x^n\in \Bbb{Z}[x]$ be a primitive polynomial with $a_0a_n\neq 0$ and $n\geq 2$. Suppose there exist a real number $\gamma\geq |a_n|$, and an index $j$ with $0\leq j\leq n-1$ such that
\begin{eqnarray*}
|a_j|&>& \sum_{0\leq i<j}|a_i|\gamma^{j-i}+\sum_{n\geq i>j}|a_i|,
\end{eqnarray*}
where the summation over empty set is defined to be zero. Then $f$ is a product of at most $n-j$ irreducible polynomials in $\mathbb{Z}[x]$. In particular, if $j=n-1$, then $f$ is irreducible in $\mathbb{Z}[x]$.
\end{theorem}
The proof of this decorous result relies on the study of the location of zeros of $f$ via Rouch\'e's
Theorem, and one may recover Perron's criterion by taking $j = n-1$ and
$a_n = 1$. Here one of the main contributions is that for indices $j$ other than $n-1$, even if we cannot reach a conclusion on the
irreducibility of $f$, one may at least find an estimate on the maximum number of irreducible factors
of $f$. In particular, Theorem \ref{J-R-1a} yields the following corollary for the case when one of the coefficients of $f$ is
divisible by a sufficiently large prime power.
\begin{corollary}[Singh and Garg \cite{J-R-1}]\label{J-R-1b}Let $f = p^sa_{j-1}x^{j-1}+p^N a_jx^j+\sum_{i=0;~i\neq j-1,j}^n a_i x^i\in \Bbb{Z}[x]$ be a primitive polynomial with $a_0a_{j-1}a_ja_n\neq 0$, $N\geq 1$, $n\geq 2$, $s\geq 0$, $1\leq j\leq n-1$, let $p$ be a prime number, $p\nmid |a_{j-1}a_j|$ such that
\begin{eqnarray*}
p^N|a_j|&>& |a_na_{j-1}|p^{2s}+\sum_{i\neq 1;i=0}^{j-1}|a_n^{i}a_{j-i}|p^{is}+\sum_{i=j+1}^n|a_i|.
\end{eqnarray*}
Then the number of irreducible factors of $f$ in $\mathbb{Z}[x]$ is at most $n-j$. In particular, if $j=n-1$, then $f$ is irreducible in $\mathbb{Z}[x]$.
\end{corollary}
The above result complements Theorem \ref{Bonciocate2} due to Bonciocat et al. \cite{Bonciocat2}. Another main result of \cite{J-R-1} generalizes Theorems \ref{Bonciocat4c} and \ref{Bonciocat4d}, which we state as follows.
\begin{theorem}[Singh and Garg \cite{J-R-1}]\label{J-R-1c}
Let $f = a_0 + a_1x +\cdots+a_n x^n\in \Bbb{Z}[x]$ be a primitive polynomial. Suppose there exist positive real numbers $\alpha$ and $\beta$ with $\alpha<\beta$ and an index $j\in \{0,\ldots,n\}$ for which
\begin{eqnarray*}
|a_j|\alpha^j &>& \sum_{0\leq i<j}|a_i|\alpha^i+(\alpha/\beta)^{j}\sum_{n\geq i>j}|a_i|\beta^i,
\end{eqnarray*}
where the summation over empty set is defined to be zero.  Further, if there exist natural numbers $n$ and $d$  satisfying $\beta-d\geq n\geq \alpha+d$ such that either $|f(n)/d|$ is a prime, or $|f(n)/d|$ is a prime power coprime to $|f'(n)|$, then $f$ is irreducible in $\mathbb{Z}[x]$.
\end{theorem}
In \cite{P-S}, Panitopol and \c{S}tef\u{a}nescu constructed some irreducibility criteria which determine a new class of univariate as well as multivariate irreducible polynomials, which are motivated from  a result of Ehrenfeucht \cite{Eh}. Although these results (see \cite{P-S}) are enthralling, yet we shall restrict our attention to univariate polynomials having integer coefficients. In this regard, we have the following result.
\begin{theorem} [Panitopol and \c{S}tef\u{a}nescu \cite{P-S}]\label{PS1a}
Let $f = a_0 + a_1x +\cdots+a_n x^n\in \Bbb{Z}[x]$ be such that
\begin{eqnarray*}
|a_0|>|a_1|+|a_2|+\cdots+|a_{n}|.
\end{eqnarray*}
If $a_0$ is prime, or $\sqrt{|a_0|}-\sqrt{|a_n|}<1$, then $f$ is irreducible in $\mathbb Z[x]$.
\end{theorem}
It is immediate from Theorem \ref{PS1a} that for any natural numbers $m$, $n$, $a$, and odd prime $p$, the polynomial
\begin{eqnarray*}
ax^m\pm x^n+a+2
\end{eqnarray*}
is irreducible in $\mathbb Z[x]$.
The connection with the location of zeros is however still critical in establishing the irreducibility as is evident from the following result, the proof of which is available in \cite[Proposition 3]{P-S}.
\begin{theorem}[Panitopol and \c{S}tef\u{a}nescu \cite{P-S}]\label{P-Sa}
If $f=a_0+a_1x+\cdots+a_nx^n\in \mathbb Z[x]$ is such that there exists a positive integer $m>1+\max_{1\leq i\leq n}\{\re{(\theta_i)}\}$, where $\theta_1,\cdots,\theta_n$ are all zeros of $f$, and $f(m)$ is prime, then $f$ is irreducible in $\mathbb Z[x]$.
\end{theorem}
This paves way for the irreducibility of Hurwitz polynomial over integers.
We note that in view of Theorem \ref{P-Sa} if $m>1+\max_{1\leq i\leq n}\{|\theta_i|\}$ and $f(m)$ is prime, then $f$ is irreducible in $\mathbb Z[x]$. In the same article, another alluring result for polynomials over integers of degree at least 9 is worth mentioning, which is stated as follows.
\begin{theorem}[Panitopol and \c{S}tef\u{a}nescu \cite{P-S}]\label{P-Sb}
Let $f=a_0+a_1x+\cdots+a_nx^n\in \mathbb Z[x]$  be such that $n\geq 9$. Suppose there exist unique polynomials $g, h\in \mathbb Z[x]$ such that
\begin{eqnarray*}
f(x)=g(x^2)+xh(x^2),~F(x)=g(x)^2-xh(x)^2.
\end{eqnarray*}
If there are $\lfloor(n+1)/2\big\rfloor$ integer values of $m$ bounded above by $-2$ for which $f(m)$ is a prime number, then $f$ is irreducible in $\mathbb Z[x]$.
\end{theorem}
If $p_1,\ldots,p_n$ are $n$ distinct prime numbers, then taking $h(x)=1$ and $g(x)=-(p_1+x)(p_2+x)\cdots (p_n+x)$ in Theorem \ref{P-Sb}, the irreducibility of the polynomial
\begin{eqnarray*}
x-(p_1+x^2)(p_2+x^2)\cdots(p_n+x^2),~n\geq 5
\end{eqnarray*}
is amazingly expeditious.
\section{Truncated Binomial Polynomials}\label{Sec:5}
An acquaintance to the binomial expansion $(1+x)^m$, the truncated polynomial $P_{m,n}$, where
\begin{eqnarray*}
P_{m,n}(x)=1+\binom{m}{1}x^1+\binom{m}{2}x^2+\cdots+\binom{m}{n}x^n, m\geq n-1
\end{eqnarray*}
is quite interesting and the connection to irreducibility is well comprehended. Note that if $n$ is not a prime, then the polynomial $P_{m, m-1}$ is reducible over $\mathbb Q$. On the other hand, if $p$ is prime, then the polynomial $P_{p, p-1}$ is irreducible over $\mathbb{Q}$ via Eisenstein's irreducibility criterion applied to the reciprocal polynomial $P_{\text{rev}}$, where $P_{\text{rev}}(x)=x^{p-1}P_{p,p-1}(x)$. The polynomial $P_{m,2}$ has negative discriminant and so, it is irreducible. The inquisitiveness regarding the irreducibility of such polynomials is quite appealing, an account of which has been given in \cite{F-K-P}, wherein to establish the irreducibility of the polynomial $P_{m,n}$, the primary concern is to investigate and strategically find the conditions under which the polynomial
\begin{eqnarray*}
F_{m, n}=\displaystyle\sum_{j=0}^{n}a_jc_jx^j,~c_j=\binom{m}{j}\binom{m-j-1}{n-j}(-1)^{n-j}
\end{eqnarray*}
is irreducible, where $a_j\in \mathbb Z\setminus \{0\}$ have prime factors not exceeding $n$. Since $P_{m, n}(x-1)=F_{m,n}(x)$ for
$a_0=a_1=\cdots=a_n=1$, an important observation is that if the polynomial $F_{m, n}$ is irreducible for all combinations of $a_0,\cdots,a_n$, then so is the truncated polynomial $P_{m, n}$. In this direction, we mention the following results which are self radiant.
\begin{theorem}[Filaseta et al. \cite{F-K-P}]\label{F-K-Pa}
The truncated binomial $P_{m,n}$ is irreducible over $\mathbb{Q}$ for all $n\leqslant 100,~n+2\leqslant m$.   
\end{theorem}
The following inferences are explicit criteria for the irreducibility of $F_{m, n}$.
\begin{theorem}[Filaseta et al. \cite{F-K-P}]
If there exists a prime $p>n$ that exactly divides $m(m - n)$, then $F_{m, n}$ is irreducible for every
choice of integers $a_0, a_1,\ldots, a_n$ with each having all of its prime factors $\leq n$.
\end{theorem}
\begin{theorem}[Filaseta et al. \cite{F-K-P}]
Let $n\geq 3$ be a fixed integer. There exists $m_0 = m_0(n)$ such that if $m \geq m_0,$ then $F_{m,n}$ is
irreducible for every choice of integers $a_0, a_1,\ldots, a_n$ with each having all of its prime factors not exceeding $n$.
\end{theorem}
Applaudingly in \cite{D-S}, the polynomial $P_{m, n}$ has been proved to be irreducible for each $n\leq 6$ and every positive integer $m\geq n+2$.
\section{Schur-type and P\'{o}lya-type irreducible polynomials}\label{Sec:6}
This section is dedicated to some coruscating irreducible criteria which have witnessed cogent extensions and generalizations over the period. We begin the mathematical expedition with irreducibility criteria of Schur-type and P\'{o}lya-type polynomials. In 1908, Schur \cite{Sch} intrigued the irreducibility of the polynomial
\begin{eqnarray*}
1\pm (x-a_1)(x-a_2)\cdots(x-a_n),
\end{eqnarray*}
where $a_0,a_1,\ldots,a_n$ are distinct rational integers and a year later projected whether the polynomial
\begin{eqnarray*}
1+{\Big\{(x-a_1)(x-a_2)\cdots(x-a_n)\Big\}}^{2^k}
\end{eqnarray*}
is irreducible for all integers $k\geq 1$. In 1919, P\'{o}lya discovered the following exhilarating  irreducibility criterion \cite{Po}.
\begin{theorem}[P\'olya \cite{Po}]
If $f\in \mathbb Z[x]$ is a polynomial of degree $n(\geq 7)$ such that there exist $n$ values $a\in \mathbb Z$ for which
\begin{eqnarray*}
0<|f(a)|<2^{-N}N!,
\end{eqnarray*}
where $N=\lceil{n/2}\rceil$, then the polynomial $f$ is irreducible over $\mathbb Q$.
\end{theorem}
It must be mentioned here that following a different approach, P\'{o}lya proved the irreducibility of $f$ having odd degree $\geq 17$ and $|f(x)|=p$ for $n$ distinct integral arguments $p$, where $p$ is a rational prime. P\'{o}lya's irreducibility criterion for  integer polynomials having small positive absolute values at several distinct integers was based on a lemma envisaged through interpolation theory due to Lagrange \cite[Lemma 3.1]{G-H-T}.

With $(a)_k$ denoting the product $a(a+1)\cdots(a+k-1)$, the results proved in \cite{G-H-T} are stated as follows.
\begin{theorem}[Gy\"ory et al.\cite{G-H-T}]\label{G-H-Ta}
Let $f\in \mathbb Z[x]$ be a polynomial of degree $n>1$. Let $0<k<n$. If there exist $k+1$ distinct integers $a$ satisfying
\begin{eqnarray*}
0<|f(a)|<2^{1-k}((n-k)/2)_k,
\end{eqnarray*}
then $f$ has no factor of degree $k$ over $\mathbb Q$.
\end{theorem}
\begin{theorem}[Gy\"ory et al. \cite{G-H-T}]\label{G-H-Tb}
Let $f\in \mathbb Z[x]$ be a polynomial of degree $n\geq 8$. Let $N\leq l<n, n-l\leq k\leq l$. If there exist $l+1$ distinct integers $a$ such that
\begin{eqnarray*}
0<|f(a)|<2^{1-N}((n-N)/2)_N,
\end{eqnarray*}
then $f$ has no factor of degree $k$ over $\mathbb Q$.
\end{theorem}
The following irreducibility result due to Levit \cite[Theorem 2]{L} is an immediate consequence of Theorem \ref{G-H-Tb} on taking  $l=n-1$.
\begin{theorem}[Gy\"{o}ry et al. \cite{G-H-T}]\label{G-H-Tc}
Let $f\in \mathbb Z[x]$ be a polynomial of degree $n\geq 1$. Let $N$ be as in Theorem \ref{G-H-Tb}. If there exist $n$ distinct values $a\in \mathbb Z$ for which
\begin{eqnarray*}
0<|f(a)|<2^{1-N}((n-N)/2)_N,
\end{eqnarray*}
then $f$ is irreducible over $\mathbb Q$.
\end{theorem}
Now we summarize some Schur-type results on irreducibility, and to state these, some notations are imperative. Let $\tau(a)$ denote the number of positive divisors of a nonzero integer $a$. Further, we observe that for any real number $\alpha$
\begin{eqnarray*}
\alpha-1<\lfloor{\alpha}\rfloor\leq \alpha\leq \lceil\alpha\rceil<\alpha+1.
\end{eqnarray*}
\begin{theorem}[Gy\"{o}ry et al. \cite{G-H-T}]\label{Shtype1}
Let $c$ and $n$ be two nonzero integers satisfying $n>2\tau(c)(2+\lfloor{\log_2|c|}\rfloor)$. Let $f(x)=(x-a_1)\ldots(x-a_m)$, where $a_i$ are distinct rational integers. Let $g(x)=h(x)f(x)+c$, where $h\in \mathbb{Z}[x]$.  Then every divisor of the polynomial $g$ is of the form $f(x)h(x)+c_1$, where $c_1$ is an integer dividing $c$.
\end{theorem}
\begin{corollary}[Gy\"{o}ry et al. \cite{G-H-T}]\label{Shtype2}
Under the hypothesis of Theorem \ref{Shtype1}, the polynomial $g$ is reducible over $\mathbb Q$ if and only if the polynomial $h$ can be expressed as
\begin{eqnarray*}
h=h_1h_2f+c_2h_1+c_1h_2,
\end{eqnarray*}
for some nonzero polynomials $h_1, h_2\in \mathbb Z$ and $c_1, c_2$ are integers with $c_1c_2=c$.
\end{corollary}
The following result of Dorwart and Ore (see \cite{Se}) is an immediate consequence of Corollary \ref{Shtype2}.
\begin{corollary} [Dorwart and Ore \cite{D-O}]\label{dorwart1}
Under the hypothesis of Theorem \ref{Shtype1}, the following assertions hold true:
\begin{enumerate}
\item If $\deg(h)<n$, then $g$ is irreducible.
\item If $\deg(h)=n$ and $g$ is reducible over $\mathbb Q$, then $h=af+b$, where $a, b$ are nonzero integers.
\end{enumerate}
\end{corollary}
\section{Irreducibility Criteria for Polynomial Shifting}
Recently in \cite{Du1}, Dubickas discussed the irreducibility of a polynomial which was shifted by an exponent of another polynomial. More precisely for any $f\in \mathbb Z[x]$ and any prime number $p$, there exists a polynomial $g\in \mathbb Z[x]$ for which the polynomial $h$ with $h(x)=f(x) - g(x)^p$ is irreducible over $\mathbb Q$.
\begin{theorem}[Dubickas \cite{Du1}]\label{Du1a}
Let $p \geq 2$ be a prime number. For each
$f \in \mathbb Z[x]$, there exists $g\in \mathbb Z[x]$ such that the polynomial $h$ with
$h(x)=f(x) - g(x)^p$ is irreducible over $\mathbb Q$.
\end{theorem}
The conclusion of Theorem \ref{Du1a} does not hold incase the prime $p$ is replaced by a composite number $\geq 2$. However, the following results were proved in \cite{Du1}.
\begin{theorem}[Dubickas \cite{Du1}]\label{Du1b}
Let $m \geq 2$ be an integer, and let $f \in\mathbb Z[x]$ be a
polynomial which is not of the form $ah(x)^k$ with integers
$a \neq 0, k \geq 2$, and $h \in \mathbb Z[x]$. Then there exists  a polynomial $g \in \mathbb Z[x]$ such
that the polynomial $F$ with $F(x)=f(x) - g(x)^m$ is irreducible over $\mathbb Q$.
\end{theorem}
Given two relatively prime polynomials with coefficients in a unique factorization domain, to decide whether their sum is irreducible or not, is in general a difficult problem, and no general answer in this respect is available. The problem seems to be a little bit easier if linear combinations of two relatively prime polynomials $f$ and $g$ in $\mathbb{Z}[x]$ are considered in the form $n_1f + n_2g$, instead of their sum, where $n_1,n_2\in\mathbb{Z}$. Such a linear combination turns out to be irreducible, provided some conditions on
the factorization of $n_1$ and $n_2$ are satisfied. In this respect, several recent results
provide irreducibility criteria for polynomials of the form $f + pg$, where $f$ and $g$ are relatively prime polynomials with rational coefficients, and $p$ is a sufficiently large prime number. In \cite{B-B-C-M}, the following elegant result is proved.
\begin{theorem}[Bonciocat et al. \cite{B-B-C-M}]\label{Bonciocat2013a}
Let $f, g \in \mathbb Z[x]$ be two relatively
prime polynomials with $\deg g = n$ and $\deg f = nd, d \geq 1$. Then for any prime
number $p$ that divides none of the leading coefficients of $f$ or $g$, and any positive
integer $k$ prime to $d$ such that
\begin{eqnarray*}
p^k\geq \Big\{2+\dfrac{1}{2^{n+1-d}H_g^{n+1}}\Big\}^{n+1-d}H_fH_g^n-\dfrac{H_f}{H_g},
\end{eqnarray*}
the polynomial $f+p^kg$ is irreducible over $\mathbb Q$, where $H_f$ and $H_g$ denote the heights of the polynomials $f$ and $g$, respectively.
\end{theorem}
Using Theorem \ref{Bonciocat2013a}, it can be verified that
for any odd prime $p$, the polynomial
\begin{eqnarray*}
1+x^{p-1}+p^{p-1}(1+x+x^2+\cdots+x^{p-1})
\end{eqnarray*}
is irreducible.

In coherence to the results mentioned in this survey, we conclude with the remark that Bonciocat et al. \cite{Bonciocat2022} have devised new irreducibility criteria for polynomials having integer coefficients whose zeros lie outside some angular sectors or outside lens shaped regions in the complex plane. 

\end{document}